\documentclass[12pt]{amsart}
\usepackage[dvips]{graphicx}
\usepackage{amsmath,graphics,xypic}
\usepackage{amsfonts,amssymb}
\theoremstyle{plain}
\newtheorem{theorem*}{Theorem}
\newtheorem*{lemma*}{Lemma}
\newtheorem{corollary*}{Corollary}
\newtheorem*{proposition*}{Proposition}
\newtheorem{conjecture*}{Conjecture}
\newtheorem{theorem}{Theorem}[section]
\newtheorem{lemma}[theorem]{Lemma}
\newtheorem{corollary}[theorem]{Corollary}

\theoremstyle{remark}

\newtheorem*{remark}{Remark}

\newtheorem*{example}{Example}
\newtheorem*{claim}{Claim}

\theoremstyle{definition}

\def\gl{\mbox{GL}} \def\Q{\Bbb{Q}} \def\F{\Bbb{F}} \def\Z{\Bbb{Z}}  \def\C{\Bbb{C}}
    
 \def\a{\alpha} \def\g{\gamma}  \def\bp{\begin{pmatrix}}
\def\sm{\setminus} \def\ep{\end{pmatrix}} \def\bn{\begin{enumerate}} 
   \def\en{\end{enumerate}}
\def\ba{\begin{array}} \def\ea{\end{array}}  
 \def\S{\Sigma}  \def\a{\alpha}  \def\ti{\tilde}
  \def\im{\mbox{Im}} 
  
\def\ker{\mbox{Ker}}\def\be{\begin{equation}} \def\ee{\end{equation}} 
   
 \def\hom{\mbox{Hom}}  
 \def\aut{\mbox{Aut}}  
 \def\dim{\operatorname{dim}}  
 \def\ct{\C[t^{\pm 1}]}    \def\ck{\C^k} \def\ckt{\C^k[t^{\pm 1}]} \def\ft{\F[t^{\pm 1}]}\def\fkt{\F^k[t^{\pm 1}]}
    
   \def\fr12{\frac{1}{2}} \def\z12{\Z[\fr12]}

\def\ol{\overline}

\def\max{\mbox{max}}
\def\min{\mbox{min}}

\def\v{\varphi}

\def\S{\Sigma}

\begin{document}

\title{Taut sutured manifolds and twisted homology}

\author{Stefan Friedl}
\address{Mathematisches Institut\\ Universit\"at zu K\"oln\\   Germany}
\email{sfriedl@gmail.com}

\author{Taehee Kim}

\address{Department of Mathematics\\
  Konkuk University\\
  Seoul 143--701\\
  Republic of Korea}
  \email{tkim@konkuk.ac.kr}

\date{\today}



\begin{abstract}
We give a necessary and sufficient criterion for a sutured manifold $(M,\gamma)$ to be taut in terms of the twisted homology of the pair $(M,R_-)$.
\end{abstract}

\maketitle

%=====================================
\section{Introduction}

A sutured manifold  $(M,\gamma)$
is a compact, connected, oriented $3$-manifold $M$ together with a set of disjoint annuli  $\g$ on $\partial M$
which  turn $M$ naturally into a cobordism between  oriented surfaces $R_-=R_-(\gamma)$ and $R_+=R_+(\gamma)$ with boundary.
We refer to Section \ref{section:bsm} for the precise definition.

We say that a sutured manifold $(M,\gamma)$ is  \emph{balanced}
if  $\chi(R_+) = \chi(R_-)$.
Balanced sutured manifolds arise in many different contexts. For example  $3$-manifolds cut along non-separating surfaces naturally give rise to balanced sutured manifolds.
For the remainder of the introduction we will only be concerned with balanced sutured manifolds. Later on we will also consider the case of
general sutured manifolds.

Given a surface $S$ with connected components $S_1\cup\dots \cup S_k$ we define its \emph{complexity} to be
$\chi_-(S)=\sum_{i=1}^k \max\{-\chi(S_i),0\}$. Following Gabai \cite[Definition~2.10]{Ga83}
we say that  a balanced sutured manifold $(M,\gamma)$ is \emph{taut} if $M$ is irreducible and if $R_-$ and $R_+$ have minimal complexity among all surfaces representing the homology class $[R_-]=[R_+]\in H_2(M,\gamma;\Z)$.

Given a representation $\a\colon \pi_1(M)\to \gl(k,\F)$ over a field $\F$ we can consider the twisted homology
groups $H_*^\a(M,R_-;\F^k)$. In this paper we give a
  necessary and sufficient
 criterion for a balanced sutured manifold $(M,\gamma)$ to be taut in terms of the twisted homology of the pair $(M,R_-)$.
More precisely, we will prove the  following  theorem:

\begin{theorem}\label{mainthmintro}\label{mainthm}
Let $(M,\gamma)$ be an irreducible  balanced sutured manifold with $M\ne S^1\times D^2$ and $M\ne D^3$. Then $(M,\gamma)$ is taut if and only if
$H_1^\a(M,R_-;\C^k)=0$ for some unitary representation $\a\colon \pi_1(M)\to U(k)$.
\end{theorem}

To the best of our knowledge both directions of the theorem are new.
Whereas the `if' direction can be proved using classical methods
the `only if' direction uses the recent revolutionary work by Agol \cite{Ag08}, Liu \cite{Liu11}, Przytycki-Wise \cite{PW12}
and Wise \cite{Wi12}.

The paper is organized as follows:
In Section \ref{section:tn} we recall the definition of the Thurston norm, which will play an important role in the
proof of Theorem \ref{mainthm}. In Section \ref{section:bsm} we introduce balanced sutured manifolds
and in Section \ref{section:twihom} we define twisted homology groups.
The `if' direction of Theorem \ref{mainthm} is proved in Section \ref{section:mainthmonlyif}
and the `only if' direction of Theorem \ref{mainthm} is proved in Section \ref{section:mainthmif}.

\subsection*{Conventions and notations.}
All $3$-manifolds are assumed to be oriented,  compact and connected, unless it says explicitly otherwise. By $\F$ we will always mean a field.

%=====================================
\section{Definitions}

%=====================================
\subsection{The Thurston norm}\label{section:tn}

Let $N$ be a 3-manifold.
It is well--known that  any class in $H^1(N;\Z)$ is dual to a properly embedded oriented surface.
The   Thurston norm of $\phi\in H^1(N;\Z)$ is defined as
\[ \|\phi\|_T := \min\{\chi_-(\S)\, |\, \S \subset N\mbox{ properly embedded and dual to }\phi\}.\]
 Thurston
\cite{Th86} showed that  $\|-\|_T$ is a seminorm on $H^1(N;\Z)$ which thus
can be extended to a seminorm on $H^1(N;\Q)$ which we also denote by $\|-\|_T$.

%=====================================
\subsection{Sutured manifolds} \label{section:bsm}
We now recall the notion of a sutured manifold.
Loosely speaking a sutured manifold is a cobordism between oriented surfaces with boundary.
More precisely, a sutured manifold is a $3$-manifold $M$ with non-trivial boundary
and together with a decomposition of its boundary
\[ \partial M=-R_-\cup s \times [-1,1] \cup R_+\]
into oriented submanifolds where the following conditions hold:
\bn
\item $s$ consists of oriented simple closed curves,
\item $\partial R_-=R_-\cap (s\times [-1,1])=s\times \{-1\}$ as oriented curves,
\item $\partial R_+=R_+\cap (s\times [-1,1])=s\times \{+1\}$ as oriented curves,
\item $R_-$ and $R_+$ are disjoint.
\en
We denote by $\gamma$  the union of the annuli $s\times [-1,1]$ together with an orientation of the `sutures' $s=s\times 0$.
Note that $R_+$ and $R_-$ are determined by $\g$,
following Gabai \cite{Ga83} we therefore usually denote a sutured manifold by $(M,\gamma)$
and we write $R_\pm(\gamma)=R_\pm$.

Note that the notion of sutured manifolds is due to Gabai  \cite{Ga83}, but our definition differs slightly from Gabai's definition
in so far as we do not allow `toroidal sutures' and as we restrict ourselves to connected $3$-manifolds.

We now recall the following definitions from the introduction:
Let $(M,\gamma)$ be a sutured manifold.
\bn
\item   $(M,\gamma)$ is  called \emph{balanced}
if   $\chi(R_+) = \chi(R_-)$.
(This definition agrees with the definition given by Juh\'asz \cite{Ju06} with the slight difference that we allow $R_\pm$ to have closed components
and that
 do not require that every boundary component of $M$ contains a suture.)
\item
$(M,\gamma)$ is said to be \emph{taut} if $M$ is irreducible and if $R_-$ and $R_+$ have minimal complexity among all surfaces representing
the homology class $[R_-]=[R_+]\in H_2(M,\gamma;\Z)$.
\en

The following well-known lemma says that given a taut sutured manifold $(M,\g)$ in almost all cases the surfaces $R_\pm(\g)$
are incompressible.

\begin{lemma}\label{lem:tautinc}
Let $(M,\gamma)$ be a taut sutured manifold. Then one of the following holds:
\bn
\item $R_-$ and $R_+$ are incompressible, or
\item $M=S^1\times D^2$ and $\g$ is a union of annuli of the form $[-1,1]\times \partial D^2$, or
\item $M=D^3$ and $M$ has at least two sutures.
\en
\end{lemma}

We conclude this section with two examples of balanced
sutured manifolds.
\bn
\item
 Let $R$ be an oriented surface, then
 $(R\times [-1,1],\partial R\times [-1,1])$ is a sutured manifold with $R_\pm=R\times \{\pm 1\}$.
 We refer to this sutured manifold as a \emph{product sutured manifold}.
\item
Let $Y$ be a   3--manifold with trivial or toroidal boundary.
Let $R \subset Y$ be a properly embedded compact oriented surface such that $Y\sm R$ is connected. Furthermore, if $Y$ has non-empty toroidal boundary we assume that $R$  intersects each boundary torus at least once in a homologically essential curve.
We pick a tubular neighborhood $R\times [-1,1]$ such that if $Y$ has non-empty boundary then we have $\partial (R\times [-1,1])\cap \partial Y=\partial R\times [-1,1]$.
We then define
 \[ Y (R) = (Y\sm R\times (-1,1),\partial Y\sm (\partial R\times (-1,1)))\]
 which is clearly a balanced sutured manifold.
 Note that $R_\pm=R\times \{\pm 1\}$.
If  $Y$ is furthermore irreducible, then it is well-known that  $R$ is Thurston norm minimizing in $Y$ if and only if $Y(R)$ is taut.
(The `only if' direction is obvious and the `if' direction follows from \cite[Corollary~5.3]{Ga83} and \cite[Corollary~2]{Th86}.)
\en

%==============================================
\subsection{Twisted homology groups} \label{section:twihom}
We recall  the definition of twisted homology groups and discuss some of their properties.
More information can for example be found in a previous paper by the authors \cite{FK06}.

Let $X$ be a topological space, $Y\subset X$ a (possibly empty) subset and $x_0\in X$ a point.
Let $\a\colon \pi_1(X,x_0)\to \gl(k,\F)$ be a representation.
This naturally induces a left $\Z[\pi_1(X,x_0)]$--module structure on $\F^k$.

Denote by $\ti{X}$ the set of all homotopy classes of paths starting at $x_0$ with the usual
topology. Then the evaluation map $p\colon \ti{X}\to X$ is the universal cover of $X$. Note that $g\in
\pi_1(X,x_0)$ naturally acts on $\ti{X}$ on the right by precomposing any path by $g^{-1}$.

Given $Y\subset X$ we let $\ti{Y}=p^{-1}(Y)\subset \ti{X}$. Then the above $\pi_1(X,x_0)$ action on
$\ti{X}$ gives rise  to a right $\Z[\pi_1(X,x_0)]$--module structure on the chain groups
$C_*(\ti{X}),C_*(\ti{Y})$ and $C_*(\ti{X},\ti{Y})$. Therefore we can form the tensor product over
$\Z[\pi_1(X,x_0)]$ with $\F^k$, we define
\[ \ba{rcl}
H_i^{\a}(X;\F^k)&=&H_i(C_*(\ti{X})\otimes_{\Z[\pi_1(X,x_0)]}\F^k)),\\
H_i^\a(Y\subset X;\F^k)&=&H_i(C_*(\ti{Y})\otimes_{\Z[\pi_1(X,x_0)]}\F^k),\\
H_i^{\a}(X,Y;\F^k)&=&H_i(C_*(\ti{X},\ti{Y})\otimes_{\Z[\pi_1(X,x_0)]}\F^k)). \ea\]

 Note that if we have
inclusions $Z\subset Y\subset X$ then we get an inclusion induced map $H_i^\a(Z\subset X;\F^k)\to
H_i^\a(Y\subset X;\F^k)$. Also note that we have an exact sequence of complexes
\[ 0 \to C_i(\ti{Y})\otimes_{\Z[\pi_1(X,x_0)]}\F^k\to
C_i(\ti{X})\otimes_{\Z[\pi_1(X,x_0)]}\F^k\to C_i(\ti{X},\ti{Y})\otimes_{\Z[\pi_1(X,x_0)]}\F^k\to
0\] which gives rise to a long exact sequence \be \label{longexact} \dots \to H_i^{\a}(Y\subset
X;\F^k)\to H_i^{\a}(X;\F^k)\to H_i^{\a}(X,Y;\F^k)\to \dots .\ee

Now denote by $Y_i,i\in I,$ the path connected components of $Y$. Pick base points $y_i\in Y_i,i\in
I,$ and paths $\g_i:[0,1]\to X$ with $\g_i(0)=y_i$ and $\g_i(1)=x_0$. Then we  get induced
representations $\a_i(\g_i):\pi_1(Y_i,y_i)\to \pi_1(X,y_i)\to \pi_1(X,x_0)\to \gl(k,\F)$ and induced
homology groups $H_j^{\a_i(\g_i)}(Y_i;\F^k)$ using the universal cover of $Y_i$.

In \cite{FK06} we proved the following lemma:

\begin{lemma} \label{lem:twiiso}
Given $\g_i$ there exists a canonical  isomorphism
\[ H_j^{\a_i}(Y_i\subset X;\F^k)\cong  H_j^{\a_i(\g_i)}(Y_i;\F^k).\]
\end{lemma}

The isomorphism type of $H_j^\a(X,Y;\F^k)$ does not depend on the choice of the
base point. In most situations we can and will therefore suppress the base point in the notation
and the arguments. We will also normally write $\a$ instead of $\a(\g_i)$. Furthermore we write
$H_j^\a(Y;\F^k)$ for $\oplus_{i\in I}H_j^\a(Y_i;\F^k)$. With these conventions  the long exact sequence
(\ref{longexact}) induces a long exact sequence
\[ \dots \to H_j^\a(Y;\F^k)\to H_j^\a(X;\F^k) \to H_j^\a(X,Y;\F^k)\to \dots.\]

For the proof of Theorem \ref{mainthm} we will need the following well-known results on twisted homology groups:

\begin{lemma} \label{lem:h0surjects}
Let $(X,Y)$ be a  pair of  spaces  with $X$ path connected and with $Y\ne \emptyset$.
Let  $\a\colon \pi_1(X)\to \gl(k,\F)$ be a representation. Then
\[  H_0(X,Y;\F^k)=0.\]
\end{lemma}

For the reader's convenience we provide a quick proof of Lemma \ref{lem:h0surjects}.

\begin{proof}
Let $A$  be a group and let $\a\colon A\to \gl(k,\F)$ be a representation.
Let  $\varphi\colon B\to A$ be a
group homomorphism. By \cite[Section~VI.3]{HS97} we obtain the  following  commutative diagram of exact sequences
\[  \ba{cccccccccc}
0&\to&\{ \a(\varphi(b))v-v | b\in B, v\in \F^k\}&\to&\F^k&\to& H_0^{\a \circ\v}(B;\F^k)&\to&0\\
&&\downarrow&&\downarrow&&\downarrow \\
0&\to&\{ \a(a)v-v | a\in A, v\in \F^k\}&\to&\F^k&\to& H_0^\a(A;\F^k)&\to&0. \ea
\]
Note that  the vertical map on the left is
injective. It follows that $H_0^{\a \circ \varphi}(B;\F^k)\to H_0^\a(A;\F^k)$ is surjective.
The lemma is an immediate consequence of this fact.
\end{proof}

We also recall the following  well-known duality theorem  (see e.g. \cite[Theorem~2.1]{CF10} and \cite[Lemma~2.3]{FK06} for a proof).

\begin{theorem}\label{thm:duality}
Let $X$ be an $n$--manifold together with a decomposition $\partial X=Y_1\cup Y_2$ where $Y_1$ and $Y_2$ are submanifolds of $\partial X$ with $\partial Y_1=\partial Y_2$.
Let  $\a\colon \pi_1(X)\to \gl(k,\F)$ be a representation.
We denote by $\a^\dagger$ the representation which is given by $(\a^\dagger)(g):=\a(g^{-1})^t$ for $g\in \pi_1(X)$.
 Then
\[
H_{n-i}^\a(X,Y_1;\F^k)\cong {H_{i}^{\a^\dagger}(X,Y_2;\F^k)}.
\]
\end{theorem}

We obtain the following corollary:

\begin{corollary}\label{cor:h0surjects}
Let $M$ be a $3$-manifold, $S$ a non-trivial proper subsurface of $\partial M$. Let $\a\colon \pi_1(M)\to \gl(k,\F)$ be a representation, then
\[ H_0^\a(M,S;\F^k)= H_3^\a(M,S;\F^k)=0.\]
\end{corollary}

\begin{proof}
It follows immediately from  Lemma \ref{lem:h0surjects} that $H_0^\a(M,S;\F^k)=0$. Let $T$ be the closure of $\partial M\sm S$. We then apply Theorem \ref{thm:duality} to $\partial M=S\cup T$ and we see that
\[ H_3^\a(M,S;\F^k)\cong H_0^{\a^\dagger}(M,T;\F^k),\]
which in turn is zero by Lemma \ref{lem:h0surjects}.
\end{proof}

Given a pair of CW-complexes $(X,Y)$ and a representation $\a\colon \pi_1(X)\to \gl(k,\F)$  we now write
\[ b_i^\a(X,Y;\F^k)=\dim H_i^\a(X,Y;\F^k) \mbox{ and } \chi^\a(X,Y;\F^k)=\sum_i (-1)^ib_i^\a(X,Y;\F^k).\]
When $\a$ is understood we will drop it from the notation.
A standard argument (see e.g. \cite{FK06}) shows the following  lemma.

\begin{lemma} \label{lemmab1} \label{lem:euler}
Let $(X,Y)$ be a pair of CW-complexes. Let  $\a\colon \pi_1(X)\to
\gl(k,\F)$ be a representation, then
\[ \chi^\a(X,Y;\F^k)=k\chi(X,Y). \]
\end{lemma}

%=====================================
\subsection{Twisted invariants and product manifolds}
If $(M,\gamma)$ is a product sutured manifold, then it is obvious that  $H_1(M,R_\pm;\F^k)=0$ for any representation $\a\colon \pi_1(M)\to \gl(k,\F)$.
We will now see that the converse to that statement is an easy consequence of work of Long and Niblo \cite{LN91}.

\begin{lemma}\label{lem:product}
Let $(M,\gamma)$ be a taut sutured manifold
with $M\ne D^3$. If $(M,\gamma)$ is not a product sutured manifold, then there exists a unitary representation
$\a\colon \pi_1(M)\to U(k)$ such that
$H_1(M,R_-;\C^k)\ne 0$.
\end{lemma}

\begin{proof}
Let $(M,\gamma)$ be a taut sutured manifold.
First note that if  $R_-$ is not connected, then $H_1(M,R_-;\Z)\ne 0$,
i.e. the trivial representation already gives rise to non-trivial homology.

Now suppose that $R_-$ is connected.
It follows from Lemma \ref{lem:tautinc}  that either $M=S^1\times D^2$ and $R_-=I\times \partial D^2$
or $R_-$ is incompressible.
In the former case it is clear that $H_1(M,R_-;\C)\ne 0$, i.e. the trivial representation has the desired property.

If $(M,\g)$ is not a product sutured manifold, then it follows from \cite[Theorem~10.5]{He76} that the map $\pi_1(R_-)\to \pi_1(M)$ is not surjective.
By \cite[Theorem~1]{LN91} the subgroup $\pi_1(R_-)\subset \pi_1(M)$ is separable,
in particular there exists an epimorphism
$\a\colon \pi_1(M)\to G$ to a finite group such that $\a(\pi_1(R_-))$ is strictly contained in $\a(\pi_1(M))$. We then consider the long exact sequence
\[ \to H_1(M,R_-;\C[G])\to H_0(R_-;\C[G])\to H_0(M;\C[G]) \to 0.\]
Note that
\[ \dim H_0(R_-;\C[G])=\frac{|G|}{|\im(\pi_1(R_-)\to G)|}>\frac{|G|}{|\im(\pi_1(M)\to G)|}=H_0(M;\C[G]).\]
It thus follows that  $H_1(M,R_-;\C[G])\ne 0$. Finally note that the representation $\a\colon \pi_1(M)\to G\to \aut(\C[G])\cong \gl(|G|,\C)$ is unitary with respect to the standard basis of $\C[G]$. The representation $\a$ thus has the desired property.
\end{proof}

\begin{remark}
Let $N$ be a $3$-manifold with empty or toroidal boundary and let $R\subset N$ be a Thurston norm minimizing surface.
We write $M=N\sm \nu R$.
Note that $R$ is a fiber of a fibration $N\to S^1$ if and only if $N(R)=(M,\gamma)$ is a product sutured manifold. Then the following hold:
\bn
\item By Lemma \ref{lem:product} the twisted invariants of $(M,\gamma)$ corresponding to representations of $\pi_1(M)$ can
detect whether $(M,\gamma)$ is a product sutured manifold.
\item
In \cite{FV11,FV12b} it is shown that twisted Alexander polynomials corresponding to representations of $\pi_1(N)$ can detect whether $R$ is a fiber or not.
\en
These results are related. In fact one can show fairly easily that (2) implies (1). On the other hand (1) does not readily imply (2)
since a representation of $\pi_1(M)$ does not necessarily extend to a representation of $\pi_1(N)$.
\end{remark}

%=====================================
\section{Proof of the `if' direction of Theorem \ref{mainthm}}\label{section:mainthmonlyif}

%=====================================
\subsection{Relationship between the rank of twisted homology and complexities of surfaces}

Given a sutured manifold $(M,\g)$ we say that a surface $S$ is \emph{properly embedded in $(M,\g)$} if $\partial S=S\cap \g$.
We now define the \emph{complexity  of $(M,\gamma)$} to be
\[
x(M,\gamma)=\min \{ \chi_-(S)\, | \, S \mbox{ properly embedded surface with $[S]=[R_-]\in H_2(M,\gamma;\Z)$}\}.
\]
Note that if $(M,\gamma)$ is taut, then by definition we have $x(M,\gamma)=\chi_-(R_-)=\chi_-(R_+)$.
Our main technical theorem of this section is the following result:

\begin{theorem}\label{mainthmonlyif}
Let $(M,\gamma)$ be  an irreducible sutured manifold such  that $R_\pm$ have no disk components.
Let  $\a\colon \pi_1(M)\to \gl(k,\F)$ be a representation.
Then the following inequality holds:
\[ \dim\, H_1(M,R_-;\F^k)+\dim\,  H_1(M,R_+;\F^k)\geq k\big(\chi_-(R_+)+\chi_-(R_-)-2x(M,\gamma)\big).\]
\end{theorem}

The following corollary is now a slight generalization of the `if' direction of Theorem \ref{mainthm}.

\begin{corollary}\label{maincor2}
Let $(M,\gamma)$ be  an irreducible balanced sutured manifold and let $\F$ be a field with involution. Assume there
exists a unitary representation $\a\colon \pi_1(M)\to \gl(k,\F)$  such that  $H_1(M,R_-;\F^k)=0$.  Then $(M,\gamma)$ is taut.
\end{corollary}

\begin{proof}
We start out with the following claim:

\begin{claim}
\[H_1(M,R_+;\F^k)=0.\]
\end{claim}

It follows from Theorem \ref{thm:duality} and the fact that $\a$ is a unitary representation
that
\be \label{equ:h2alpha}  H_{2}^\a(M,R_+;\F^k)=H_1^{\a^\dagger}(M,R_-;\F^k)=\ol{H_1^{\a}(M,R_-;\F^k)}=0,\ee
where given an $\F$-vector space $V$ we denote by $\ol{V}$ the vector space given by the same underlying abelian group but with involuted $\F$-multiplication.
On the other hand it follows from the assumption that $(M,\gamma)$ is balanced and from Poincar\'e duality that
\[ \chi(M)=\frac{1}{2}\chi(\partial M)=\frac{1}{2}(\chi(R_-)+\chi(R_+))=\chi(R_+).\]
This implies by Lemma \ref{lem:euler} that
\[ \chi(M,R_+;\F^k)=k\chi(M,R_+)=k(\chi(M)-\chi(R_+))=0.\]
It now follows from (\ref{equ:h2alpha}) and Corollary \ref{cor:h0surjects} that $ H_{1}(M,R_+;\F^k)=0$.
This concludes the proof of the claim.

We first  suppose that $R_-$ and $R_+$ have no disk components.
Since $\chi_-(R_\pm)\geq x(M,\gamma)$, it follows from Theorem \ref{mainthmonlyif}, the assumption that $H_1(M,R_-;\F^k)=0$ and the above claim that
\[ x(M,\gamma)=\chi_-(R_-)=\chi_-(R_+).\]
Since $M$ is furthermore assumed to be irreducible it follows that $M$ is taut.
This concludes the proof of the corollary if  $R_-$ and $R_+$ have no disk components.

We now consider the case  that $R_-$ or $R_+$ has a disk component. Without loss of generality we assume that $R_-$ has a disk component $D$.
 Since $H_1(M,R_-;\F^k)=0$ it follows from the long exact  sequence of the pair $(M,R_-)$ that the map
\[  H_0(R_-;\F^k)\to H_0(M;\F^k)\]
is injective. In particular $H_0(D;\F^k)\to H_0(M;\F^k)$ is injective.
It follows from  \cite[Section~VI.3]{HS97} that we have  the  following  commutative diagram of exact sequences
\[ \ba{cccccccccc}
0&\to&0&\to&\F^k&\to& H_0^{\a}(D;\F^k)&\to&0\\
&&\downarrow&&\downarrow&&\downarrow \\
0&\to&\{ \a(a)v-v | a\in \pi_1(M), v\in \F^k\}&\to&\F^k&\to& H_0^\a(M;\F^k)&\to&0. \ea\]
We see that the map on the right is injective only if $\a$ is the trivial representation.
We thus showed that $H_1(M,R_-;\Z)=0$. Since $M$ is connected this is only possible if $R_-$ is also connected,
i.e. $R_-$ is a disk. It follows that $\chi_-(R_-)=0$.
Now it suffices to show that $\chi_-(R_+)=0$. By the previous claim, we have $H_1(M,R_+;\F^k)=0$. Since $\chi(R_-)=1$ and $M$ is assumed to be irreducible and balanced, one can deduce that $R_+$ also has a disk component. Then using the same argument for $R_-$ as above, one can conclude that $R_+$ is a disk, and it now follows that $\chi_-(R_+)=0$.
\end{proof}

\begin{example}
Let $(M,\gamma)$ be an irreducible balanced sutured manifold such that   $H:=H_1(M;\Z)$ is a torsion-free group. We denote by $Q(H)$ the quotient field of $\Z[H]$. Note that $Q(H)$ is naturally a field with involution and that the homomorphism $\pi_1(M)\to H\to \gl(1,Q(H))$ is a unitary representation.
We can  consider the $\Z[H]$--module $H_1(M,R_-;\Z[H])$ and the
$Q(H)$-module $H_1(M,R_-;Q(H))$.
It follows from Corollary \ref{maincor2} that $(M,\gamma)$ is taut if $H_1(M,R_-;Q(H))=0$.

An alternative proof of this statement can be given using the sutured Floer homology $SFH(M,\g)$ which was introduced by Juh\'asz \cite{Ju06}.
Indeed, the following implications hold:
\[
\ba{rcl}
&& H_1(M,R_-;Q(H))=0 \\
&\Leftrightarrow& H_1(M,R_-;\Z[H]) \mbox{ is $\Z[H]$-torsion} \\
&\Leftrightarrow & \mbox{the Alexander polynomial of $(M,\gamma)$ is non-zero}\\
&\Leftrightarrow& \mbox{the Euler characteristic of $SFH(M,\gamma)$ is non-zero}\\
&\Rightarrow& \mbox{$SFH(M,\gamma)$ is non-zero}\\
&\Leftrightarrow & (M,\gamma) \mbox{ is taut}.\ea \]
Here the first two equivalences follow from elementary algebraic arguments (see e.g. \cite{Tu01}),
the third equivalence was proved in \cite{FJR11}, the fourth statement is elementary
and the last equivalence was proved by  Juh\'asz (see \cite[Theorem~9.18]{Ju06} and \cite[Theorem~1.4]{Ju08}).
\end{example}

%=====================================
\subsection{Proof of Theorem \ref{mainthmonlyif}}

We now give a proof of Theorem \ref{mainthmonlyif}.
Let $(M,\gamma)$ be  an irreducible  sutured manifold  such that $R_\pm$ have no disk components.
Let $\a\colon \pi_1(M)\to \gl(k,\F)$ be a representation.
We have to show that
\[ \dim\, H_1(M,R_-;\F^k)+\dim\,  H_1(M,R_+;\F^k)\geq k\big(\chi_-(R_+)+\chi_-(R_-)-2x(M,\gamma)\big).\]
We start out with the following claim.

\begin{claim}
There exists a properly embedded surface  $S\subset M\sm (R_-\cup R_+)$ with $\chi_-(S)=x(M,\gamma)$
  and such that $M$ cut along $S$ is the union of two disjoint (not necessarily connected)
manifolds $M_\pm$ such that $R_\pm\subset \partial M_\pm$.
\end{claim}

Let $T$ be  a properly embedded surface in $(M,\gamma)$  which realizes $x(M,\gamma)$ and which is disjoint from $R_-\cup R_+$.
We denote by  $X=\{X_i\}_{i\in I}$  the set of components of $M\sm T$.
We now define $Z$ to be the union of the closures of all $X_i$ which intersect $R_-$ non-trivially.
We define
\[S:=\ol{\partial Z \sm \partial M}.\]
Note that for an appropriate orientation of $S$ we have $\partial Z=-R_-\cup F_-\cup S$ where $F_-$ is a subsurface of $\g$.
In particular with this orientation we have that  $[S]$ is homologous to $[R_-]\in H_2(M,F_-;\Z)$.
We now write $M_-=Z$ and $M_+=\ol{M\sm Z}$. Note that $\partial M_+=-R_+\cup F_+\cup S$ where $F_+$ is a subsurface of $\g$.
Finally note that $S$ is a union of components  of $T$, hence $x(M,\gamma)\leq \chi_-(S)\leq \chi_-(T)=x(M,\gamma)$.
It is now clear that $S$ has the desired properties. This concludes the proof of the claim.

We continue with $S,M_-,M_+$ as in the claim.
We will now use the conventions of Section \ref{section:twihom}, i.e. we will write
\[ H_i^\a(S;\fkt)=H_i^\a(S\subset M;\fkt)\mbox{ and } H_i^\a(M_\pm;\fkt)=H_i^\a(M_\pm\subset M;\fkt).\]
These groups in particular are well-defined even if $S$ and $M_\pm$ are disconnected.
We will now prove the following claim.

\begin{claim}
We have the following equality
\[  k(\chi(S)-\chi(R_-))=\chi(M_-,R_-;\F^k)-\chi(M_-,S;\F^k)\]
and similarly for the ``$+$'' subscript.
\end{claim}

We will prove the claim for the ``$-$'' subscript, the other case is proved exactly the same way.
It is well-known that for any pair of spaces $(X,Y)$ we have $\chi(X,Y)=\chi(X)-\chi(Y)$.
It thus follows that
\[ \chi(M_-)=\chi(R_-)+\chi(M_-,R_-) \mbox{ and } \chi(M_-)=\chi(S)+\chi(M_-,S).\]
Subtracting these two terms we conclude that
\[ \chi(R_-)-\chi(S)=\chi(M_-,S)-\chi(M_-,R_-).\]
The claim now follows from Lemma \ref{lem:euler}.

\begin{claim}
We have the following inequality
\[
b_1(M,R_-;\F^k)\geq b_1(M_-,R_-;\F^k)-\chi(M_+,S;\F^k),\]
and similarly with the roles of ``$-$'' and ``$+$'' reversed.
\end{claim}

We  write
\[ \ba{ccl} K_-&=&\ker(H_1(R_-;\F^k)\to H_1(M_-;\F^k)), \mbox{  and }\\
K&=&\ker(H_1(R_-;\F^k)\to H_1(M;\F^k)).\ea \]
Note that $H_1(R_-;\F^k)\to H_1(M;\F^k)$ factors through $H_1(R_-;\F^k)\to H_1(M_-;\F^k)$,
it follows that  $\dim(K)\geq \dim(K_-)$.

Consider the following commutative diagram of exact sequences (with $\F^k$-coefficients understood)
\[
\ba{ccccccccccccccccccc}
0&\hspace{-0.3cm} \to \hspace{-0.3cm} & K_-& \hspace{-0.3cm} \to \hspace{-0.3cm} & H_1(R_-)&\hspace{-0.3cm} \to \hspace{-0.3cm} & H_1(M_-)&\hspace{-0.3cm} \to \hspace{-0.3cm}  & H_1(M_-,R_-) &\hspace{-0.3cm} \to \hspace{-0.3cm} &
H_0(R_-)&\hspace{-0.3cm} \to \hspace{-0.3cm} & H_0(M_-)&\hspace{-0.3cm} \to \hspace{-0.3cm} &0 \\
&&\downarrow && \downarrow && \downarrow && \downarrow && \downarrow && \downarrow && \\
0&\hspace{-0.3cm} \to \hspace{-0.3cm} & K& \hspace{-0.3cm} \to \hspace{-0.3cm} & H_1(R_-)&\hspace{-0.3cm} \to \hspace{-0.3cm} & H_1(M)&\hspace{-0.3cm} \to \hspace{-0.3cm}  & H_1(M,R_-) &\hspace{-0.3cm} \to \hspace{-0.3cm} &
H_0(R_-)&\hspace{-0.3cm} \to \hspace{-0.3cm} & H_0(M)&\hspace{-0.3cm} \to \hspace{-0.3cm} &0. \\
 \ea \]
Note that by exactness the alternating sum of the dimensions in the two horizontal  sequences are zero.
We therefore obtain that
\[
\ba{rcl}
b_1(M_-,R_-;\F^k)&=&\dim(K_-)-b_1(R_-;\F^k)+ b_1(M_-;\F^k)+b_0(R_-;\F^k)-b_0(M_-;\F^k) \\
b_1(M,R_-;\F^k)&=& \dim(K)-b_1(R_-;\F^k)+ b_1(M;\F^k)+b_0(R_-;\F^k)-b_0(M;\F^k). \ea \]
Subtracting we see that
\[
\ba{rcl}
&&b_1(M,R_-;\F^k)-b_1(M_-,R_-;\F^k)\\
&=&  \big(b_1(M;\F^k)+\dim(K)-b_0(M;\F^k)\big)-\big(b_1(M_-;\F^k)+\dim(K_-)-b_0(M_-;\F^k)\big)\\[2mm]
&\geq &
  \big(b_1(M;\F^k)-b_0(M;\F^k)\big)-\big(b_1(M_-;\F^k)-b_0(M_-;\F^k)\big).\ea \]
We now write $L=\ker(H_2(M,M_-;\F^k)\to H_1(M_-;\F^k))$.
Consider the following piece of the long exact sequence of the pair $(M,M_-)$, where  $\F^k$-coefficients are once again understood:
\[
0\to L\to H_2(M,M_-)\hspace{-0.0cm} \to \hspace{-0.0cm}  H_1(M_-)\hspace{-0.0cm} \to \hspace{-0.0cm}  H_1(M)\hspace{-0.0cm} \to \hspace{-0.0cm}  H_1(M,M_-)\hspace{-0.0cm} \to \hspace{-0.0cm}  H_0(M_-)\hspace{-0.0cm} \to \hspace{-0.0cm}  H_0(M)\hspace{-0.0cm} \to \hspace{-0.0cm}  0. \]
(Here we used Lemma \ref{lem:h0surjects} to conclude that $H_0(M,M_-;\F^k)=0$.)
We now use that the alternating sum of the dimensions in the above exact sequence is zero to deduce that
\[ \ba{rcl}
 &&\big(b_1(M;\F^k)-b_0(M;\F^k)\big)-\big(b_1(M_-;\F^k)-b_0(M_-;\F^k)\big)\\
 &=& b_1(M,M_-;\F^k)-b_2(M,M_-;\F^k)+\dim( L)\\
 &\geq &  b_1(M,M_-;\F^k)-b_2(M,M_-;\F^k)\\
 &= &  b_1(M_+,S;\F^k)-b_2(M_+,S;\F^k)\\
 &=&-\chi(M_+,S;\F^k).\ea \]
Here  the second to last equality follows from excision and the last equality
 follows from Corollary  \ref{cor:h0surjects} applied to $(M_+,S)$. This concludes the proof of the claim.

Finally note that by Corollary  \ref{cor:h0surjects} we have the following inequalities
\[ b_1(M_\pm,R_\pm;\F^k)\geq b_1(M_\pm,R_\pm;\F^k)-b_2(M_\pm,R_\pm;\F^k)= \chi(M_\pm,R_\pm;\F^k).\]
 Combining these inequalities with the above claims we obtain that
\[ \ba{rcl}
&&b_1(M,R_-;\F^k)+b_1(M,R_+;\F^k)\\
&\geq & b_1(M_-,R_-;\F^k)-\chi(M_+,S;\F^k)+b_1(M_+,R_+;\F^k)-\chi(M_-,S;\F^k)\\
&\geq & \chi(M_-,R_-;\F^k)-\chi(M_+,S;\F^k)+\chi(M_+,R_+;\F^k)-\chi(M_-,S;\F^k)\\
&=& (\chi(M_-,R_-;\F^k)-\chi(M_-,S;\F^k))\,+\,(\chi(M_+,R_+;\F^k)-\chi(M_+,S;\F^k))\\
&=& k(2\chi(S)-\chi(R_+)-\chi(R_-)).\ea \]
We can now conclude the proof of Theorem \ref{mainthmonlyif}.
Recall that we assumed that $M$ is irreducible and $R_\pm$ have no disk components. This implies that no component of $R_\pm$ is a sphere, and hence $\chi_-(R_\pm)=-\chi(R_\pm)$.
Finally note that for any surface $T$ we have $\chi_-(T)\geq -\chi(T)$.
Combining these observations with the above inequality we see  that
\[ \ba{rcl}
b_1(M,R_-;\F^k)+b_1(M,R_+;\F^k)&\geq& k(2\chi(S)-\chi(R_+)-\chi(R_-))\\
&\geq& k(\chi_-(R_+)+\chi_-(R_-)-2\chi_-(S))\\
&=& k(\chi_-(R_+)+\chi_-(R_-)-2x(M,\gamma)).\ea \]
This concludes the proof of Theorem \ref{mainthmonlyif}.

%=====================================
\section{Proof of the `only if' direction of Theorem \ref{mainthm}}\label{section:mainthmif}

In this section we will prove the following theorem, which is a slight strengthening of
the  `only if' direction of Theorem \ref{mainthm}.

 \begin{theorem}\label{mainthmif}
Let $(M,\gamma)$ be a  taut sutured manifold with $M\ne S^1\times D^2$ and $M\ne D^3$.
Then there exists a unitary representation $\a\colon \pi_1(M)\to U(k)$ such
that
\[ H_*(M,R_-;\C^k)=H_*(M,R_+;\C^k)=0.\]
 \end{theorem}

\begin{remark}
We will now argue that  we indeed have to exclude the cases $M=S^1\times D^2$ and $M=D^3$.
\bn
\item
Suppose that $M=S^1\times D^2$ and that $\gamma$ consists of two  meridional sutures, i.e. of the form
\[ \gamma=[a,b]\times \partial D^2\cup [c,d]\times \partial D^2\]
with $[a,b]$ and $[c,d]$ having disjoint images in $S^1=[0,1]/0\sim 1$ where we endow the central `sutures' of the annuli with opposite orientations. In that case $R_+$ is also of the form $I\times \partial D^2$. It is straightforward to see that
there exists no representation $\a\colon \pi_1(M)=\Z\to U(k)$ such
that $H_1(M,R_-;\C^k)=0$.
\item An elementary argument furthermore shows that $M=D^3$ with at least two sutures also does not admit
a unitary representation $\a\colon \pi_1(M)\to U(k)$ with
$H_1(M,R_-;\C^k)=H_1(M,R_+;\C^k)=0$.
\en
\end{remark}

%=====================================
\subsection{The results of Agol, Liu, Przytycki-Wise and Wise} \label{section:apw}

Before we can state the results of Agol, Liu, Przytycki-Wise and Wise we first need to introduce a few more  definitions.
\bn
\item
A group $\pi$ is called \emph{residually finite rationally solvable} or \emph{RFRS} if there
exists a filtration  of groups $\pi=\pi_0\supset \pi_1 \supset \pi_2\dots $
such that the following hold:
\bn
\item $\cap_i \pi_i=\{1\}$,
\item   $\pi_i$ is a normal, finite index subgroup of  $\pi$ for any $i$,
\item for any $i$ the map $\pi_i\to \pi_i/\pi_{i+1}$ factors through $\pi_i\to H_1(\pi_i;\Z)/\mbox{torsion}$.
\en
We refer to \cite{Ag08} for details.
\item Given a sutured manifold $(M,\gamma)$ the double $DM=D(M,\g)$ is defined to be the double of $M$ along $R=R_+\cup R_-$, i.e. $DM=M\cup_{R_+\cup R_-}M$. Note that $DM$ is a $3$-manifold with empty or toroidal boundary.
We denote by $r\colon DM\to M$ the retraction map given by `folding' the two copies of $M$ along $R$.
\item Let $N$ be a $3$-manifold with empty or toroidal boundary.
An integral class $\phi \in H^1(N;\Z)=\hom(\pi_1(N),\Z)$ is called \emph{fibered}  if  there exists a fibration $p\colon N\to S^1$ such that
$\phi=p_*\colon \pi_1(N)\to \Z$. We say $\phi\in H^1(N;\Q)$ is \emph{fibered} if a non-trivial integral multiple of $\phi$ is fibered.
We say $\phi\in H^1(N;\Q)$ is \emph{quasi-fibered} if any open neighborhood of $\phi$ contains a fibered class.
\en

The statement of the following theorem is  explicitly stated  in the proof of   \cite[Theorem~6.1]{Ag08}.

\begin{theorem} \textbf{\emph{(Agol)}} \label{thm:agol}
Let $(M,\gamma)$ be a taut sutured manifold with $M\ne S^1\times D^2$.
We write $W=D(M,\g)$ and we denote by $\phi\in H^1(W;\Z)$ the Poincar\'e dual of $[R_-]\in H_2(W,\partial W;\Z)$.
If $\pi_1(M)$ is  RFRS,
then there exists an epimorphism $\a\colon \pi_1(M)\to G$ to a finite  group,
such that in the covering $p\colon \widetilde{W}\to W$
corresponding to $\a\circ r_*\colon \pi_1(W)\to G$ the  class $p^*(\phi)\in H^1(\widetilde{W};\Z)$
is quasi-fibered.
\end{theorem}

We can apply Agol's theorem in our context due to the following theorem:

\begin{theorem} \textbf{\emph{(Liu, Przytycki-Wise, Wise)}}\label{thm:lpw}
Let $M$ be an irreducible 3-manifold with non-trivial boundary, then $\pi_1(M)$ is virtually RFRS,
i.e. $\pi_1(M)$ admits a finite index subgroup which is RFRS.
\end{theorem}

In the hyperbolic case this theorem  is due to Wise \cite{Wi12}, in the graph manifold case it is due to Liu \cite{Liu11} and independently to Przytycki-Wise \cite{PW11} and in the remaining (`mixed') case it is due to Przytycki-Wise \cite{PW12}.
We refer to the survey paper \cite{AFW12} for background and more precise references.

%=====================================
\subsection{Twisted Alexander polynomials and the Thurston norm}

Let $N$ be a 3-manifold with empty or toroidal boundary, let $\phi\in H^1(N;\Z)$ be non-trivial
and let $\a\colon \pi_1(N)\to \gl(k,\F)$ be a representation.
Then given $i\in \{0,1,2\}$ we can consider the $i$-th twisted Alexander polynomial $\Delta_{N,\phi,i}^\a\in \ft$,
which is defined as the order of the twisted Alexander module $H_i(N;\fkt)$.
Note that $\Delta_{N,\phi,i}^\a\in \ft$ is well-defined up to multiplication by a unit in $\ft$.
In this paper we will not give a definition of twisted Alexander polynomials,
but we refer to the foundational papers
\cite{Lin01,Wa94,KL99} and the survey paper \cite{FV10} for details and more information.

Given $p(t)\ne 0 \in \ft$ we can write $p(t)=\sum_{i=r}^sa_it^i$ with $a_r\ne 0, a_s\ne 0$ and we define
\[ \deg(p(t)):=s-r.\]
Note that $\deg \Delta_{N,\phi,i}^\a$ is well-defined.
We can now state the following theorem which  is implicit in the proof of \cite[Theorem~5.9]{FV12a}.

\begin{theorem}\label{thm:fv}
Let $W$ be a 3-manifold with empty or toroidal boundary. Suppose that $W\ne S^1\times S^2$ and $W\ne S^1\times D^2$.
Let $\phi\in H^1(W;\Z)$.
Suppose there exists a finite cover $f\colon \widetilde{W}\to W$ such that $f^*(\phi)$ is quasi-fibered.
Then there exists a unitary representation $\a\colon \pi_1(W)\to U(k)$ such that
$\Delta_{W,\phi,i}^\a\ne 0$ for $i=0,1,2$ and such that
\[ k\cdot \|\phi\|_T =\sum_{i=0}^2 (-1)^{i+1} \deg \Delta_{W,\phi,i}^\a.\]
\end{theorem}

%=====================================
\subsection{Proof of Theorem \ref{mainthmif}}

We are now in a position to prove Theorem \ref{mainthmif}.
Let $(M,\gamma)$ be a taut sutured manifold such that  $M\ne S^1\times D^2$ and $M\ne D^3$.

Since $M\ne D^3$ and $M$ is irreducible, it follows that $R_-$ and $R_+$ do not have spherical components.
It follows in a straightforward way from our assumption that
$M\ne S^1\times D^2,  D^3$ and from Lemma \ref{lem:tautinc}
that $R_-$ and $R_+$ do not have disk components.
We thus see that $\chi_-(R_\pm)=-\chi(R_\pm)$.

We write $W=D(M,\g)$ and we denote by $\phi\in H^1(W;\Z)$ the class dual to $R_-$.
Since $M$ is taut it follows from \cite[Corollary~5.3]{Ga83} and \cite[Corollary~2]{Th86} that
 $R_-$ is a Thurston norm minimizing representative of $\phi$.

By Theorem \ref{thm:lpw} there exists a  finite cover  $p\colon M'\to M$  such that $\pi_1(M')$ is RFRS. Note that this finite cover gives canonically rise to a finite cover $(M',\gamma')\to (M,\gamma)$
of sutured manifolds.
Also note that $(M',\g')$ is  again a taut sutured manifold.
(This follows for example
from \cite[Corollary~5.3]{Ga83} since pull-backs of taut foliations are obviously taut.)

We now write  $W'=D(M',\gamma')$ and we denote by $\phi'\in H^1(W';\Z)$ the Poincar\'e dual of $[R_-']\in H_2(W',\partial W';\Z)$.
Note that the finite covering $p\colon M'\to M$ gives rise to a finite covering map $W'\to W$  which we also denote by $p$.
We furthermore note that $\phi'=p^*(\phi)$. By Theorem \ref{thm:agol}
there  exists an epimorphism $\a\colon \pi_1(M')\to G$ to a finite  group
such that in the covering $q\colon \widetilde{W'}\to W'$
corresponding to $\a\circ r_*\colon \pi_1(W')\to G$ the  class $q^*(\phi')$ is quasi-fibered.

 Note that $p\circ q\colon \widetilde{W'}\to W$ is a finite covering map such that
the pull back of $\phi$ is quasi-fibered.
By Theorem \ref{thm:fv} there  exists a unitary representation $\a\colon \pi_1(W)\to U(k)$ such that
$\Delta_{W,\phi,i}^\a\ne 0$ for $i=0,1,2$ and such that
\be \label{equ:fv} k\cdot \|\phi\|_T =\sum_{i=0}^2 (-1)^{i+1} \deg \Delta_{W\phi,i}^\a.\ee
We now write $X_-:=W\sm \nu R_-$ and $X_+:=W\sm \nu R_+$.
Note that $R_-$ admits two inclusions into $X_-$ which we denote
by $\iota_l$ (`left embedding') and $\iota_r$ (`right embedding').
Similarly we denote the  two inclusions of $R_+$ into $X_+$  by $\iota_l$ and $\iota_r$.

\begin{claim}
The inclusion induced  maps
 \[ \ba{rcl} \iota_l,\iota_r\colon H_i(R_-;\ck)&\to& H_i(X_-;\ck)\mbox{ and }\\
  \iota_l,\iota_r\colon H_i(R_+;\ck)&\to& H_i(X_+;\ck)\ea \]
    are isomorphisms for all $i$.
\end{claim}

Since  $\Delta_{W,\phi,1}^{\a}\ne 0$ it follows from \cite[Proofs~of~Propositions~3.5~and~3.6]{FK06} that
the maps
\be \label{equ:02iso} \iota_{l},\iota_r\colon H_0(R_-;\ck)\to H_0(X_-;\ck) \mbox{ and } \iota_l,\iota_r\colon H_2(R_-;\ck)\to H_2(X_-;\ck)\ee
are isomorphisms. Note that $W$ has empty or toroidal boundary. A standard Poincar\'e duality argument shows that $\chi(W)=0$.
 This implies in turn that $\chi(R_-)=\chi(X_-)$. It now follows from
(\ref{equ:02iso}) and from Lemma \ref{lem:euler}  that
\be \label{equ:sameb1} b_1(R_-;\ck)=b_1(X_-;\ck).\ee
It remains to show that an isomorphism is given by the inclusion induced maps.

By  \cite[Proposition~3.2]{FK06} we have the following long exact sequence in homology:
\be \label{equ:ses} \dots H_i(R_-;\C^k)\otimes \ct\xrightarrow{\iota_l-\iota_rt} H_i(X_-;\C^k)\otimes \ct\to H_i(W;\ckt)\to \dots \ee
Since $\Delta_{W,\phi,i}^\a\ne 0$ for $i=0,1,2$ it follows that $H_i(W;\ckt)$ is $\ct$--torsion for any $i$.
On the other hand $H_i(R_-;\C^k)\otimes \ct$ is $\ct$-torsion free, it follows that
 the long exact sequence (\ref{equ:ses}) splits into short exact sequences:
\[ 0\to  H_i(R_-;\C^k)\otimes \ct\xrightarrow{\iota_l-\iota_rt} H_i(X_-;\C^k)\otimes \ct\to H_i(W;\ckt)\to 0.\]
By the definition of the order of a module we have
\be \label{equ:deltai} \Delta_{W,\phi,i}^\a=\det\left(\iota_l-\iota_rt: H_i(R_-;\C^k)\otimes \ct\to H_i(X_-;\C^k)\otimes \ct\right).\ee
(Here we implicitly use the fact we established above that $b_i(R_-;\ck)=b_i(X_-;\ck)$ for $i=0,1,2$.)
Now recall that if $A$ and $B$ are $s\times s$-matrices over $\C$, then
\be \label{equ:detab} \deg(\det(A+tB))=s\Leftrightarrow \det(A)\ne 0 \mbox{ and }\det(B)\ne 0.\ee
Since $\iota_l,\iota_r\colon  H_i(R_-;\ck)\to H_i(X_-;\ck) $ are isomorphisms for $i=0$ and $i=2$ we thus obtain from (\ref{equ:deltai}) that
\be \label{equ:02} \deg \Delta_{W,\phi,i}^\a=b_i(R_-;\C^k)\mbox{ for }i=0,2.\ee
Note that by Lemma \ref{lem:euler} and by the above we have
\[ \sum_{i=0}^2 (-1)^{i+1} b_i(R_-;\C^k)=-k\chi(R_-)=k\chi_-(R_-)=k\cdot \|\phi\|_T .\]
Combining this with (\ref{equ:fv}), (\ref{equ:deltai}) and (\ref{equ:02})  we see that
\[ \deg\left(\det(\iota_l-\iota_rt: H_1(R_-;\C^k)\otimes \ct\to H_1(X_-;\C^k)\otimes \ct)\right)=b_1(R_-;\C^k).\]
By   (\ref{equ:detab})  we now conclude
the maps $ \iota_l,\iota_r\colon H_1(R_-;\ck)\to H_1(X_-;\ck)$ are isomorphisms.
This concludes the proof of the claim.

Now note that we have the following commutative diagram:
 \[ \xymatrix{ H_i(R_-;\ck)\ar[rr]^\cong\ar[dr]&& H_i(X_-;\ck)\\  &H_i(M;\ck).\ar[ur]}\]
 It follows that that the map $H_i(R_-;\ck)\to H_i(M;\ck)$ is injective. (The same conclusion also holds with ``$+$'' subscripts.)
Hence in order to show that the maps  $H_i(R_-;\ck)\to H_i(M;\ck)$ are isomorphisms it suffices to show the
following claim:

\begin{claim}
For $i=0,1,2$ we have
\[ b_i(R_-;\ck)=b_i(M;\ck).\]
\end{claim}

We first record  that by the discussion preceding the claim  we have
\be \label{equ:birm}  b_i(M;\ck)\geq b_i(R_-;\ck)\mbox{ and } b_i(M;\ck)\geq b_i(R_+;\ck).\ee
Now note that we can write $X_+=M_1\cup_{R_-}M_2$ where $M_1,M_2$ are two copies of $M$. We consider the  long exact Mayer-Vietoris sequence corresponding to this
decomposition:
\be \label{equ:les2}\dots \to H_i(R_-;\ck)\to H_i(M_1;\ck)\oplus H_i(M_2;\ck)\to H_i(X_+;\ck) \to\dots \ee
By the above inclusion induced maps $H_i(R_-;\ck)\to H_i(M_1;\ck)$ and $H_i(R_-;\ck)\to  H_i(M_2;\ck)$
are injective. We thus see that the long exact sequence (\ref{equ:les2}) splits into short exact sequences:
\[ 0  \to H_i(R_-;\ck)\to H_i(M_1;\ck)\oplus H_i(M_2;\ck)\to H_i(X_+;\ck) \to 0.\]
This implies that
\[ b_i(X_+;\ck)=b_i(M_1;\ck)+b_i(M_2;\ck)-b_i(R_-;\ck)\mbox{ for }i=0,1,2.\]
It follows from the equality $b_i(X_+;\ck)=b_i(R_+;\ck)$ and from (\ref{equ:birm})  that
\be \label{equ:bim} b_i(M_1;\ck)\geq b_i(R_+;\ck)=b_i(X_+;\ck)=b_i(M_1;\ck)+(b_i(M_2;\ck)-b_i(R_-;\ck)).\ee
Note that  the second summand is again non-negative by (\ref{equ:birm}).
We now see from (\ref{equ:bim})  that in fact
 $b_i(M;\ck)=b_i(M_2;\ck)=b_i(R_-;\ck)$ as claimed.
Similarly we prove that $b_i(M_1;\ck)=b_i(R_+;\ck)$.
This concludes the proof of the claim and thus the proof of Theorem \ref{mainthmif}.

\subsection*{Acknowledgements}
The second author was supported by Basic Science Research Program through the National Research Foundation of Korea (NRF) funded by the Ministry of Education, Science and Technology (No. 2012001747 and No. 20120000341).

\end{document}